\documentclass{article}

\usepackage[all]{xy}
\usepackage[english]{babel}
\usepackage{amsmath}
\usepackage{amsthm}
\usepackage{amssymb}

\newcommand{\comments}[1]{} 

\newtheorem{thm}{Theorem}

\theoremstyle{definition}

\theoremstyle{remark}
\newtheorem{rem}{Remark}

\theoremstyle{remark}

\def\cM{\mathcal M}

\def\cP{\mathcal P}
\def\cS{\mathcal S}
\def\cC{\mathcal C}

\def\cH{\mathcal H}

\def\bR{\mathbb R} 

\def\bB{\mathbb B}
\def\bS{\mathbb S}
\def\bN{\mathbb N}

\def\bC{\mathbb C}

\def\Ker{\operatorname{Ker}}

\def\clifford{\cC\ell}

\def\La{\Lambda}
\def\lz{\langle}
\def\pz{\rangle}
\def\pa{\partial}

\def\pod{\underline}
\def\nad{\overline}

\def\spin{\mathfrak{spin}}

\def\La{\Lambda}
\def\la{\lambda}

\begin{document}

\title{Complete orthogonal Appell systems for spherical monogenics}

\author{R. L\' avi\v cka
\thanks{
Faculty of Mathematics and Physics, Charles University, Sokolovsk\'a 83, 186 75 Praha 8, Czech Republic;
Email: \texttt{lavicka@karlin.mff.cuni.cz} Tel: +420 221 913 204  Fax: +420 222 323 394 }}

\date{}

\maketitle

\begin{abstract}
In this paper, we investigate properties of Gelfand-Tsetlin  bases mainly for spherical monogenics, that is, 
for spinor valued or Clifford algebra valued homogeneous solutions of the Dirac equation in the Euclidean space. 
Recently it has been observed that in dimension 3 these bases form an Appell system.  
We show that Gelfand-Tsetlin bases of spherical monogenics  form complete orthogonal Appell systems in any dimension. 
Moreover, we study the corresponding Taylor series expansions for monogenic functions. 
We obtain analogous results for spherical harmonics as well.

\medskip\noindent{\bf Keywords:}  spherical harmonics, spherical monogenics, Gelfand-Tsetlin basis, Appell system, orthogonal basis, Taylor series

\medskip\noindent{\bf MSC classification:} 30G35
\end{abstract}

\section{Introduction}

The main aim of this paper is to study properties of the so-called Gelfand-Tsetlin bases for spherical harmonics and, above all,
for spherical monogenics in the Euclidean space $\bR^m$, that is, for spinor valued or Clifford algebra valued monogenic polynomials in $\bR^m$. Monogenic functions are just solutions of the equation $\pa F=0$ where the Dirac operator $\pa$ is defined as
\begin{equation}\label{Dirac}
\pa=e_1\frac{\pa\ }{\pa x_1}+\cdots+e_m\frac{\pa\ }{\pa x_m}.
\end{equation}
On the one hand, monogenic functions are a~higher dimensional analogue of holomorphic functions of one complex variable. On the other hand, the Dirac operator $\pa$ factorizes the Laplace operator $\Delta$ in the sense that $\Delta=-\pa^2$ and so the theory of monogenic functions which is nowadays called Clifford analysis refines harmonic analysis.  

As is well-known, we can expand a~given holomorphic function $f$ on the unit disc $\bB_2$ into its Taylor series
$$f(z)=\sum_{k=0}^{\infty}\frac{f^{(k)}(0)}{k!}\;z^k.$$
The coefficients of this Taylor series are expressed directly by the complex derivatives of the function $f$ at the origin due to the fact that
$(z^k)'=k z^{k-1}$. In general, we say that basis elements possess the Appell property or form an Appell system if their derivatives are equal to 
a~multiple of another basis element.   
Moreover, the powers $z^k$ form an orthogonal basis for holomorphic functions in $L^2(\bB_2,\bC)$, the space of square-integrable functions on $\bB_2$. 

In this paper, we suggest a~proper analogue of the powers $z^k$ for monogenic functions in any dimension. 
For a~detailed account of history of this topic, we refer to \cite{BGLS}. 
Let us only remark that S. Bock and K.~G\"urlebeck described orthogonal Appell bases for quaternion valued monogenic
functions in $\bR^3$ and $\bR^4$ (see \cite{BG}, \cite{Bock2009},
\cite{Bock2010c}). 
In \cite{lavSL2}, these bases in dimension 3 are constructed in another way.
The first construction of orthogonal bases for spherical monogenics even in any dimension was given by F. Sommen, see \cite{som,DSS}.
In \cite[pp. 254-264]{DSS}, orthogonal bases for spherical monogenics in $\bR^{p+q}$ are constructed when these orthogonal bases are known in $\bR^p$ and $\bR^q$. The construction is based on solving a~Vekua-type system of partial differential equations.  In \cite{van, step2}, these bases are interpreted as $Spin(p)\times Spin(q)$-invariant orthogonal bases and are obtained using extremal projections. It turns out that the Gelfand-Tsetlin bases corresponds to the case when $p=1$. 
Moreover, in dimension 3,  another constructions of orthogonal bases of spherical monogenics using the standard bases of spherical harmonics were done also by I. Ca\c c\~ao, S. Bock, K. G\"urlebeck and H.~Malonek (see \cite{cac}, \cite{CacGueMal}, \cite{BockCacGue}, \cite{CacGueBock}). 
In \cite{BGLS}, 
it is observed that the complete orthogonal Appell system constructed in \cite{BG} can be considered as a~Gelfand-Tsetlin basis.
Actually, the main aim of this paper is to show that Gelfand-Tsetlin bases for spherical monogenics form complete orthogonal Appell systems in any dimension. 
In \cite{BGLS}, it was shown that, in dimension 3, elements of the Gelfand-Tsetlin bases for spinor valued spherical monogenics possess  the Appell property not only with respect to one (the last) variable but even with respect to all three variables.
Moreover, Gelfand-Tsetlin bases have been intensively studied in other settings as well, see \cite{ckH, GTBH} for Hermitean Clifford analysis and \cite{DLS, DLS2, lav_isaac09, DLS4} for Hodge-de Rham systems.

In this paper, we show that Gelfand-Tsetlin bases form complete orthogonal Appell systems  for spherical harmonics (see Section 2), for Clifford algebra valued spherical monogenics (see Section 3) and, finally, for spinor valued spherical monogenics (see Section 4). In each of these cases, we recall the Gelfand-Tsetlin construction of orthogonal bases and study the corresponding Taylor series expansions. At the end of Section 2, an abstract definition of Gelfand-Tsetlin bases for spin modules is given.  

\section{Spherical harmonics}

In this section, we construct a~complete orthogonal Appell system for spherical harmonics.
Let us recall a~standard construction of orthogonal bases in this case. 
Denote by $\cH_k(\bR^m)$ the space of complex valued harmonic polynomials in $\bR^m$ which are $k$-homogeneous. 
Let $(e_1,\ldots,e_m)$ be an orthonormal basis of the Euclidean space $\bR^m$.
Then the construction of an orthogonal basis for the space $\cH_k(\bR^m)$ is based on the following decomposition (see \cite[p. 171]{GM})
\begin{equation}\label{branchHarm}
\cH_k(\bR^m)=
\bigoplus_{j=0}^k F^{(k-j)}_{m,j}\cH_{j}(\bR^{m-1})
\end{equation}
which is orthogonal with respect to the $L^2$-inner product, say, on the unit ball $\bB_m$ in $\bR^m$.
Here the embedding factors $F^{(k-j)}_{m,j}$ are defined as the polynomials  
\begin{equation}\label{EFHarm}
F^{(k-j)}_{m,j}(x)=\frac{(j+1)_{k-j}}{(m+2j-2)_{k-j}}\; |x|^{k-j} C^{m/2+j-1}_{k-j}(x_m/|x|),\ x\in\bR^m
\end{equation} 
where $x=(x_1,\ldots, x_m)$, $|x|=\sqrt{x_1^2+\cdots+x_m^2}$ and $C^{\nu}_k$ is the Gegenbauer polynomial given by
\begin{equation}\label{gegenbauer}
C^{\nu}_k(z)=\sum_{i=0}^{[k/2]}\frac{(-1)^i(\nu)_{k-i}}{i!(k-2i)!}(2z)^{k-2i}\text{\ \ with\ \ }
(\nu)_{k}=\nu (\nu+1)\cdots (\nu+k-1).
\end{equation}
The decomposition \eqref{branchHarm} shows that spherical harmonics in $\bR^m$ can be easily expressed in terms of spherical harmonics in $\bR^{m-1}$.
Indeed, for each $P\in\cH_k(\bR^m)$, we have that
$$P(x)=P_k(\pod x)+F^{(1)}_{m,k-1}(x) P_{k-1}(\pod x)+\cdots+F^{(k)}_{m,0}(x) P_{0}(\pod x),\ x=(\pod x, x_m)\in\bR^m$$
for some uniquely determined polynomials $P_j\in\cH_{j}(\bR^{m-1})$. Of course, here $F^{(0)}_{m,k}=1$ and $\pod x=(x_1,\ldots, x_{m-1})$.

Applying the decomposition \eqref{branchHarm}, we easily construct an orthogonal basis of the space $\cH_k(\bR^m)$ by induction on the dimension $m$. Indeed, as the polynomials $(x_1\mp ix_2)^k$ form an orthogonal basis of the space $\cH_k(\bR^2)$ an orthogonal basis of the space $\cH_k(\bR^m)$ is formed by the polynomials
\begin{equation}\label{GTHarm}
h_{k,\mu}(x)=(x_1\mp ix_2)^{k_2}\prod^m_{r=3}F^{(k_r-k_{r-1})}_{r,k_{r-1}}
\end{equation}
where $\mu$ is an arbitrary sequence of integers $(k_{m-1}, \ldots, k_3,\pm k_2)$ such that $k=k_m\geq k_{m-1}\geq\cdots\geq k_3\geq k_2\geq 0$.
Furthermore, we have taken the normalization of the embedding factors $F^{(k-j)}_{m,j}$ so that the basis elements $h_{k,\mu}$ possess the following Appell property.

\begin{thm}\label{tAPHarm}
Let $m\geq 3$ and  let $h_{k,\mu}$ be the basis elements of the spaces $\cH_k(\bR^m)$ defined in \eqref{GTHarm} with $\mu=(k_{m-1},\ldots, k_3, \pm k_2)$.
Then we have that
\begin{itemize}
\item[(i)] $\pa_{x_m}  h_{k,\mu}=0$ for $k=k_{m-1}$;
\item[(ii)] $\pa_{x_m}  h_{k,\mu}=k\;  h_{k-1,\mu}$ for $k>k_{m-1}$;
\item[(iii)] $\pa_{\pm}^{k_2}\;\pa^{k_3-k_2}_{x_3}\cdots\pa^{k-k_{m-1}}_{x_m}  h_{k,\mu}=k!$ where $\pa_{\pm}=(1/2)(\pa_{x_1}\pm i\pa_{x_2})$.
\end{itemize}
\end{thm}

\begin{proof}
The statement (i) follows from the fact that $F^{(0)}_{m,j}=1$.
Using standard formulas for Gegenbauer polynomials (see \cite{AAR}), it is easy to verify that, for $k>j$, $\pa_{x_m} F^{(k-j)}_{m,j}=k\;F^{(k-1-j)}_{m,j}$, which implies (ii). 
Finally, we get (iii) by applying (ii) and the fact that $\pa_{\pm}\; (x_1\mp ix_2)^{k}=k\; (x_1\mp ix_2)^{k-1}$. 
\end{proof}

To summarize, we have constructed a~complete orthogonal Appell system for the complex Hilbert space
$L^2(\bB_m,\bC)\cap \Ker \Delta$ of $L^2$-integrable harmonic functions $g:\bB_m\to\bC$. Here $\bB_m$ is the unit ball in $\bR^m$. 
Indeed, we have the following result.

\begin{thm}\label{tTaylorHarm} Let $m\geq 3$ and, for each $k\in\bN_0$, denote by $N^m_k$ the set of sequences $(k_{m-1}, \ldots, k_3, \pm k_2)$ of integers  such that $k\geq k_{m-1}\geq\cdots\geq k_3\geq k_2\geq 0$. 

\begin{itemize}
\item[(a)]
Then an orthogonal basis of the space $L^2(\bB_m,\bC)\cap \Ker \Delta$ is formed by the polynomials
$h_{k,\mu}$ for $k\in\bN_0$ and $\mu\in N^m_k.$
Here the basis elements $h_{k,\mu}$ are defined in \eqref{GTHarm}.

\item[(b)]
Each function $g\in L^2(\bB_m,\bC)\cap \Ker \Delta$ has a~unique orthogonal series expansion
\begin{equation}\label{taylor_Harm}
g = \sum_{k=0}^{\infty}\sum_{\mu\in N^m_k} \mathbf{t}_{k,\mu}(g)\;h_{k,\mu} 
\end{equation}
for some complex coefficients $\mathbf{t}_{k,\mu}(g)$. 

In addition, for $\mu=(k_{m-1}, \ldots, k_3, \pm k_2)\in N^m_k$, we have that
\begin{equation}\label{coeff_Harm}
\mathbf{t}_{k,\mu}(g)
=\frac{1}{k!}\;\pa_{\pm}^{k_2}\pa^{k_3-k_2}_{x_3}\cdots\pa^{k-k_{m-1}}_{x_m}g(x)|_{x=0}
\end{equation}
with $\pa_{\pm}=(1/2)(\pa_{x_1}\pm i\pa_{x_2})$.
\end{itemize}
\end{thm}

\begin{proof}
It is well-known that the closure of
the orthogonal direct sum
$$\bigoplus_{k=0}^{\infty} \cH_k(\bR^m)$$
with respect to the $L^2$-inner product 
is just the space $L^2(\bB_m,\bC)\cap \Ker \Delta$, which 
gives ($a$). The formula \eqref{coeff_Harm} then follows directly 
from the Appell property of the basis elements, namely, from the property (iii) of Theorem \ref{tAPHarm}.  
\end{proof}

For a~function $g\in L^2(\bB_m,\bC)\cap \Ker \Delta$, we call the orthogonal series expansion \eqref{taylor_Harm} its generalized Taylor series.

\medskip
From the point of view of representation theory, the space $\cH_k(\bR^m)$ forms naturally an irreducible module over the group $SO(m)$ of rotations in $\bR^m$ when $m\geq 3$. Under the action of $SO(2)$, the module $\cH_k(\bR^2)$ decomposes as $\cH_k(\bR^2)=\lz(x_1+ix_2)^k\pz\oplus\lz(x_1-ix_2)^k\pz$.
Here $\lz M\pz$ stands for the linear span of a~set $M.$
It is well known that the so-called Spin group $Spin(m)$ is a~double cover of the group $SO(m)$ and each $SO(m)$-module can be considered as a~special representation of the group $Spin(m)$.
In particular, $\cH_k(\bR^m)$ is an irreducible module under the action of the group $Spin(m)$ defined by
$$
[h(s)(P)](x) = P(s^{-1}xs),\ s\in Spin(m)\text{\ \ and\ \ }x\in\bR^m.
$$
See \cite[Chapter 3]{GM} for details.
We show that the constructed basis \eqref{GTHarm} is actually a~Gelfand-Tsetlin basis of the $Spin(m)$-module $\cH_k(\bR^m)$.  

\paragraph{Gelfand-Tsetlin bases for spin modules}

Now we recall an abstract definition of a~Gelfand-Tsetlin basis for 
any given irreducible finite dimensional $Spin(m)$-module $V$ (see \cite{GT, mol}). 
We assume that the space $V$ is endowed with an invariant inner product.

The first step of the construction of a Gelfand-Tsetlin basis consists in reducing the symmetry to the group $Spin(m-1),$ realized as the subgroup of $Spin(m)$ describing rotations fixing the last vector $e_m.$ It turns out that, under the action of the group $Spin(m-1),$ the module $V$ is reducible and  
decomposes into a~multiplicity free direct sum of irreducible
$Spin(m-1)$-submodules
\begin{equation}\label{branch}
V=\bigoplus_{\mu_{m-1}}V(\mu_{m-1}).
\end{equation}
This irreducible decomposition is multiplicity free and so it is orthogonal.
Let us remark that, in representation theory, the decomposition \eqref{branch} is called the branching of the module $V$.

Of course, we can further reduce the symmetry to the group $Spin(m-2),$ the subgroup of $Spin(m)$ describing rotations fixing the last two vectors $e_{m-1}, e_m.$
Then each piece $V(\mu_{m-1})$ of (\ref{branch}) decomposes into irreducible $Spin(m-2)$-submodules $V(\mu_{m-1},\mu_{m-2})$ and so on. 
Hence we end up with the decomposition of the given $Spin(m)$-module
$V$ into irreducible $Spin(2)$-modules $V(\mu).$
Moreover, any such module $V(\mu)$ is uniquely determined by the sequence of labels
\begin{equation}
\label{pattern}
\mu=(\mu_{m-1},\ldots,\mu_2).
\end{equation}

To summarize, the given module $V$ is
the direct sum of irreducible $Spin(2)$-modules
\begin{equation}\label{branch+}
V=\bigoplus_{\mu} V(\mu).
\end{equation}
Moreover, the decomposition (\ref{branch+}) is obviously orthogonal. 
Now it is easy to obtain an orthogonal basis of the space $V.$
Indeed, as each irreducible $Spin(2)$-module $V(\mu)$ is one-dimensional
we  easily construct a~basis of the space $V$ by taking a~non-zero vector $e(\mu)$ from each piece $V(\mu).$ 
The obtained basis $E=\{e(\mu)\}_{\mu}$ is called a~Gelfand-Tsetlin basis of the module $V.$
By construction, the basis $E$ is orthogonal 
with respect to any invariant inner product given on the module $V$.
Moreover, each vector $e(\mu)\in E$ is uniquely determined by its index $\mu$ up to a~scalar multiple. 
In other words, for the given orthonormal basis $(e_1,\ldots,e_m)$ of $\bR^m$, the Gelfand-Tsetlin basis $E$ is uniquely determined up to a~normalization. 

\medskip
It is easily seen that, for the $Spin(m)$-module $\cH_k(\bR^m)$, the decomposition \eqref{branchHarm} is nothing else than its branching and,
consequently, the basis \eqref{GTHarm} is obviously its Gelfand-Tsetlin basis, uniquely determined by the property (iii) of Theorem~\ref{tAPHarm}.
Moreover, the Appell property described in Theorem \ref{tAPHarm} is not a~coincidence but the consequence of the fact 
that $\pa_{x_m}$ is an invariant operator under the action of the subgroup $Spin(m-1)$.

\section{Clifford algebra valued spherical monogenics}

In this section, we construct a~complete orthogonal Appell system for Clifford algebra valued spherical monogenics.
For an account of Clifford analysis, we refer to \cite{BDS, DSS, GS, GHS}.
Denote by $\clifford_m$ either the real Clifford algebra $\bR_{0,m}$ or the complex one $\bC_m$,
generated by the vectors $e_1,\ldots,e_m$ such that
$e_j^2=-1$ for $j=1,\ldots,m.$
As usual, we identify a~vector $x=(x_1,\ldots,x_m)\in\bR^m$ with the element $x_1e_1+\cdots+x_me_m$ of the Clifford algebra $\clifford_m$.

First we construct an orthogonal basis for the space $\cM_k(\bR^m,\clifford_m)$ of $k$-homogeneous monogenic polynomials $P:\bR^m\to\clifford_m$,  
endowed with a~$\clifford_m$-valued inner product
\begin{equation}\label{L2productCA}
(P,Q)_{\clifford_m}=\int_{\bB_m}\bar P Q\;d\la^m.
\end{equation}
Here $\la^m$ is the Lebesgue measure in $\bR^m$ and $a\to \bar a$ is the conjugation on $\clifford_m$ (see \cite[p. 86]{DSS}).
We want to proceed as in the harmonic case so we need to express  spherical monogenics in $\bR^m$ in terms of spherical monogenics in $\bR^{m-1}$,
which is done in the following theorem.  

\begin{thm}
The space $\cM_k(\bR^m,\clifford_m)$ has the orthogonal decomposition
\begin{equation}\label{branch_CA}
\cM_k(\bR^m,\clifford_m)=
\bigoplus_{j=0}^k X^{(k-j)}_{m,j}\cM_{j}(\bR^{m-1},\clifford_m).
\end{equation}
Here the embedding factors $X^{(k-j)}_{m,j}$ are defined as the polynomials
\begin{equation}\label{EF_CA}
X^{(k-j)}_{m,j}(x)=F^{(k-j)}_{m,j}(x)+
\frac{j+1}{m+2j-1}\;F^{(k-j-1)}_{m,j+1}(x)\;\pod x e_m,\ x\in\bR^m
\end{equation}
where $\pod x=x_1e_1+\cdots+x_{m-1} e_{m-1}$, $F^{(k-j)}_{m,j}$ are given in \eqref{EFHarm} and $F^{(-1)}_{m,k+1}=0$.
\end{thm}

\begin{proof}
See \cite[Theorem 2.2.3, p. 315]{DSS} for a~proof. 
Denote by $\cP_k(\bR^{m-1},\clifford_m)$ the space of $k$-homogeneous polynomials $P:\bR^{m-1}\to\clifford_m$.
Then, in the proof, the decomposition \eqref{branch_CA} is obtained by applying the Cauchy-Kovalevskaya extension operator $CK=e^{x_m e_m\pod \pa}$ to the Fischer decomposition of the space $\cP_k(\bR^{m-1},\clifford_m)$,
that is,
\begin{equation}\label{fischer}
\cP_k(\bR^{m-1},\clifford_m)=
\bigoplus_{j=0}^k \;(\pod x e_{m})^{k-j}\cM_{j}(\bR^{m-1},\clifford_m).
\end{equation}  
Here $\pod\pa=e_1\pa_{x_1}+\cdots+e_{m-1} \pa_{x_{m-1}}.$ 
Indeed, it holds that $$CK(\cP_k(\bR^{m-1},\clifford_m))=\cM_k(\bR^m,\clifford_m)$$ and, for each $P\in\cM_{j}(\bR^{m-1},\clifford_m)$, we have that
$$CK((\pod x e_{m})^{k-j} P(\pod x))=\mu^{(k-j)}_{m,j} X^{(k-j)}_{m,j}(x) P(\pod x)$$
where the non-zero constants $\mu^{(k-j)}_{m,j}$ are defined as 
$\mu^{(2l)}_{m,j}=(-1)^l(C_{2l}^{m/2+j-1}(0))^{-1}$ and $\mu^{(2l+1)}_{m,j}=(-1)^l\frac{m+2j+2l-1}{m+2j-2}(C_{2l}^{m/2+j}(0))^{-1}$
(see \cite[Lemma 1]{BGLS}). 
We want to have a~decomposition analogous to \eqref{branch_CA} also for spinor valued polynomials (see \eqref{branch_spinor} below) and therefore we have used 
the Fischer decomposition \eqref{fischer} given in terms of powers of $\pod x e_m$ and not $\pod x$ as usual.
Moreover, we have chosen a~different normalization of the embedding factors $X^{(k-j)}_{m,j}$ than in \cite{DSS,BGLS}, namely, we have removed the constants $\mu^{(k-j)}_{m,j}$.
\end{proof}

Using the decomposition \eqref{branch_CA}, we easily construct an orthogonal basis of the space $\cM_k(\bR^m,\clifford_m)$ by induction on the dimension $m$
as explained in \cite[pp. 262-264]{DSS}. Indeed, as the polynomial $(x_1-e_{12} x_2)^{k_2}$ forms a~basis of $\cM_{k_2}(\bR^{2},\clifford_2)$ 
an orthogonal basis of the space $\cM_k(\bR^m,\clifford_m)$ is formed by the polynomials
\begin{equation}\label{GT_CA}
f_{k,\mu}=X^{(k-k_{m-1})}_{m,k_{m-1}}X^{(k_{m-1}-k_{m-2})}_{m-1,k_{m-2}}\cdots X^{(k_3-k_2)}_{3,k_2} (x_1-e_{12} x_2)^{k_2}
\end{equation}
where $\mu$ is an arbitrary sequence of integers $(k_{m-1}, \ldots ,k_2)$ such that $k=k_m\geq k_{m-1}\geq\cdots\geq k_3\geq k_2\geq 0$.
Here $e_{12}=e_1e_2$. 
Due to non-commutativity the order of factors in the product \eqref{GT_CA} is important.
It is easy to see that the basis elements $f_{k,\mu}$ possess again the Appell property.

\begin{thm}\label{tAP_CA}
Let $m\geq 3$ and  let $f_{k,\mu}$ be the basis elements of the spaces $\cM_k(\bR^m,\clifford_m)$ defined in \eqref{GT_CA} with $\mu=(k_{m-1},\ldots, k_2)$.
Then we have that
\begin{itemize}
\item[(i)] $\pa_{x_m}  f_{k,\mu}=0$ for $k=k_{m-1}$;
\item[(ii)] $\pa_{x_m}  f_{k,\mu}=k\;  f_{k-1,\mu}$ for $k>k_{m-1}$;
\item[(iii)] $\pa_{12}^{k_2}\;\pa^{k_3-k_2}_{x_3}\cdots\pa^{k-k_{m-1}}_{x_m}  f_{k,\mu}=k!$ where $\pa_{12}=(1/2)(\pa_{x_1}+e_{12}\pa_{x_2})$.
\end{itemize}
\end{thm}

\begin{proof}
It is obvious from the fact that, for $k>j$, $\pa_{x_m} X^{(k-j)}_{m,j}=k\;X^{(k-j-1)}_{m,j}$ and $X^{(0)}_{m,j}=1$.
\end{proof}

Actually, we have constructed a~complete orthogonal Appell system for the right $\clifford_m$-linear Hilbert space
$L^2(\bB_m,\clifford_m)\cap \Ker \pa$ of $L^2$-integrable monogenic functions $g:\bB_m\to\clifford_m$.
Indeed, it is easy to show the following result.

\begin{thm}\label{tTaylor_CA} Let $m\geq 3$ and, for each $k\in\bN_0$, denote by $J^m_k$ the set of sequences $(k_{m-1}, k_{m-2},\ldots, k_2)$ of integers  such that $k\geq k_{m-1}\geq\cdots\geq k_3\geq k_2\geq 0$. 

\begin{itemize}
\item[(a)] Then an orthogonal basis of the space $L^2(\bB_m,\clifford_m)\cap \Ker \pa$ is formed by the polynomials
$f_{k,\mu}$ for $k\in\bN_0$ and $\mu\in J^m_k.$
Here the basis elements $f_{k,\mu}$ are defined in \eqref{GT_CA}.

\item[(b)]
Each function $g\in L^2(\bB_m,\clifford_m)\cap \Ker \pa$ has a~unique orthogonal series expansion
\begin{equation}\label{taylor_CA}
g = \sum_{k=0}^{\infty}\sum_{\mu\in J^m_k} f_{k,\mu}\;\mathbf{t}_{k,\mu}(g)
\end{equation}
for some coefficients $\mathbf{t}_{k,\mu}(g)$ of $\clifford_m$. 

In addition, for $\mu=(k_{m-1},\ldots,k_2)\in J^m_k$, we have that
\begin{equation*}
\mathbf{t}_{k,\mu}(g)
=\frac{1}{k!}\;\pa_{12}^{k_2}\pa^{k_3-k_2}_{x_3}\cdots\pa^{k-k_{m-1}}_{x_m}g(x)|_{x=0}
\end{equation*}
with $\pa_{12}=(1/2)(\pa_{x_1}+e_{12}\pa_{x_2})$.
\end{itemize}
\end{thm}

For a~function $g\in L^2(\bB_m,\clifford_m)\cap \Ker \pa$, we call the orthogonal series expansion \eqref{taylor_CA} its generalized Taylor series.

\medskip
In the next section, we show that the studied bases can be interpreted as Gelfand-Tsetlin bases at least for spinor valued spherical monogenics.

\section{Spinor valued spherical monogenics}

Now we adapt the results obtained in the previous section for spinor valued spherical monogenics.
Recall that the Spin group $Spin(m)$ is defined as the set of finite products of even number of unit vectors of $\bR^m$ endowed with the Clifford multiplication.
As is well known, the Lie algebra $\spin(m)$ of the group $Spin(m)$ can be realized as the space of bivectors, that is, 
$\spin(m)=\lz e_{12}, e_{13},\ldots,e_{m-1,m}\pz$ with $e_{ij}=e_i e_j$. 
Let $\bS$ be a~basic spinor representation of the group $Spin(m)$ and let $\cM_k(\bR^m,\bS)$ be the space of 
$k$-homogeneous monogenic polynomials $P:\bR^m\to\bS$. Then it is well-known that, in contrast with the space $\cM_k(\bR^m,\bC_m)$, the space $\cM_k(\bR^m,\bS)$ 
is an example of an irreducible module under the so-called $L$-action,
defined by
$$
[L(s)(P)](x) = s\,P(s^{-1}xs),\ s\in Spin(m)\text{\ \ and\ \ }x\in\bR^m.
$$
Now we recall an explicit realization of the space $\bS$.
For $j=1,\ldots,n,$ put
$$w_j=\frac 12(e_{2j-1}+ie_{2j}),\ \ \nad w_j=\frac 12(-e_{2j-1}+ie_{2j})\text{\ \ and\ \ }I_j=\nad w_jw_j.$$
Then $I_1,\ldots,I_n$ are mutually commuting idempotent elements in $\bC_{2n}.$
Moreover, $I=I_1I_2\cdots I_n$ is a~primitive idempotent in $\bC_{2n}$ and
$\bS_{2n}=\bC_{2n}I$
is a~minimal left ideal in $\bC_{2n}.$ Putting $W=\lz w_1,\ldots,w_n\pz,$ we have that
\begin{equation}
\label{spinor}
\bS_{2n}=\La(W)I,\ \  \bS^{+}_{2n}=\La^{+}(W)I\text{\ \ and\ \ } \bS^{-}_{2n}=\La^{-}(W)I
\end{equation}
where $\La(W)$ is the exterior algebra over $W$ with the even part $\La^{+}(W)$ and the odd part $\La^{-}(W).$
Putting $\theta_{2n}=(-i)^ne_1e_2\cdots e_{2n},$ we have that
\begin{equation}
\label{spm}
\bS^{\pm}_{2n}=\{u\in \bS_{2n}: \theta_{2n}u=\pm u\}.
\end{equation}
Let us recall that $\bS^{\pm}_{2n}$ are just two inequivalent basic spinor representations of the group $Spin(2n)$. 
On the other hand,
there exists only a~unique  basic spinor representation $\bS$ of the group $Spin(2n-1)$
and, as $Spin(2n-1)$-modules, the modules $\bS^{\pm}_{2n}$ are both equivalent to $\bS.$
See \cite[pp. 114-118]{DSS} for details.

Now we are going to construct explicitly a~Gelfand-Tsetlin basis for the space $\cM_k(\bR^m,\bS)$. 
First we recall the branching for spherical monogenics described in \cite{BGLS}. 
When you adapt the decomposition \eqref{branch_CA} for spinor valued polynomials you get obviously 
\begin{equation}\label{branch_spinor}
\cM_k(\bR^m,\bS)=
\bigoplus_{j=0}^k X^{(k-j)}_{m,j}\cM_{j}(\bR^{m-1},\bS).
\end{equation}
Indeed, it is easy to see that by multiplying $\bS$-valued polynomials in $\bR^{m-1}$ with the embedding factors $X^{(k-j)}_{m,j}$ from the left  
you get $\bS$-valued polynomials in $\bR^{m}$. 
In the even dimensional case $m=2n$, the decomposition \eqref{branch_spinor} describes the branching 
of the module $\cM_k(\bR^{m},\bS)$, that is, its decomposition into $Spin(m-1)$-irreducible submodules.
In the odd dimensional case $m=2n-1$, under the action of $Spin(2n-2)$, the module $\bS$ splits into two inequivalent submodules $\bS^{\pm}\simeq\bS^{\pm}_{2n-2} $ and so each module $\cM_{j}(\bR^{2n-2},\bS)$ in \eqref{branch_spinor} decomposes further as 
\begin{equation}\label{branch_spinor_odd}
\cM_{j}(\bR^{2n-2},\bS)=\cM_{j}(\bR^{2n-2},\bS^+)\oplus\cM_{j}(\bR^{2n-2},\bS^-).
\end{equation}
See \cite[Theorems 1 and 2]{BGLS} for details.

Using the decompositions \eqref{branch_spinor} and \eqref{branch_spinor_odd}, it is easy to construct Gelfand-Tsetlin bases for the module $\cM_k(\bR^m,\bS)$ by induction on the dimension $m$ as is explained already in \cite{BGLS}. 
Let $m=2n$ or $m=2n-1.$
To do this we need to describe a~Gelfand-Tsetlin basis of the space $\bS$ itself. 
The space $\bS$ is a~basic spinor representation for $Spin(m)$. 
As $Spin(2n-2)$-module, the space $\bS$ has the irreducible decomposition $\bS=\bS^+\oplus \bS^-.$ By reducing the symmetry to $Spin(2n-4)$,
the pieces $\bS^{\pm}$ themselves further decompose and so on. 
Indeed, for $j=0,\ldots,n-1,$ denote by $\cS_j$ the set of sequences of the length $j$ consisting of the signs $\pm$.  For each $\nu\in\cS_j,$ define (by induction on $j$) the subset $\bS^{\nu}$ of the set $\bS$ such that $\bS^{\emptyset}=\bS$ and, for $\nu=(\pod\nu,\pm),$ we have that $\bS^{\nu}=(\bS^{\pod\nu})^{\pm}.$ 
Put $\cS^m=\cS_{n-1}$.
Then we get the following decomposition of the space $\bS$ into irreducible $Spin(2)$-submodules
\begin{equation}\label{GTBspinors}
\bS=\bigoplus_{\nu\in\cS^m}\bS^{\nu}\text{\ \ \ with\ \ \ }\bS^{\nu}=\lz v^{\nu}\pz
\end{equation}  
where, in each 1-dimensional piece $\bS^{\nu}$, we have chosen an arbitrary non-zero element $v^{\nu}$.
The last ingredient for the construction is to describe Gelfand-Tsetlin bases for spherical monogenics in dimension 2.
Obviously, for a~given $\nu\in\cS^m$ and $k\in\bN_0$, the polynomial $(x_1-e_{12}x_2)^k v^{\nu}$ forms a~Gelfand-Tsetlin basis of $\cM_k(\bR^2,\bS^{\nu})$.
Now we are ready to prove the following theorem.

\begin{thm}\label{GT_spinor} 
Let $m\geq 3$ and let $\bS$ be a~basic spinor representation of $Spin(m).$

\begin{itemize}
\item[(i)] Then a~Gelfand-Tsetlin basis of the $Spin(m)$-module $\cM_k(\bR^{m},\bS)$ is formed by the polynomials
\begin{equation}\label{eGT_spinor}
f^{\nu}_{k,\mu}=f_{k,\mu}\; v^{\nu}
\end{equation}
where $\nu\in\cS^m$ and $\mu\in J^m_k$. 
Here $f_{k,\mu}$ are as in Theorem \ref{tTaylor_CA}, $\cS^m$ and $v^{\nu}$ as in \eqref{GTBspinors}. 

In addition, the basis \eqref{eGT_spinor} is orthogonal with respect to any invariant inner product on the module $\cM_k(\bR^{m},\bS)$, including the $L^2$-inner product and the Fischer inner product.

\item[(ii)] 
The Gelfand-Tsetlin basis \eqref{eGT_spinor} is uniquely determined by the property that, 
for each $\nu=(\pod\nu,\pm)\in\cS^m$ and $\mu=(k_{m-1}, k_{m-2},\ldots, k_2)\in J^m_k$,
\begin{equation}\label{AP_spinor} 
\pa_{\pm}^{k_2}\;\pa^{k_3-k_2}_{x_3}\cdots\pa^{k-k_{m-1}}_{x_m}  f^{\nu}_{k,\mu}=k!\;v^{\nu}. 
\end{equation}
Here $\pa_{\pm}=(1/2)(\pa_{x_1}\pm i\pa_{x_2})$.

\end{itemize}
\end{thm}

\begin{proof} The statement (i) is obvious from the construction of the basis.
Moreover, the Appell property \eqref{AP_spinor} follows directly from the statement (iii) of Theorem \ref{tAP_CA} and the fact that, 
for $\nu=(\pod\nu,\pm)$, we have that $e_{12} v^{\nu}=\pm iv^{\nu}$ and hence $(x_1-e_{12}x_2)^k v^{\nu}=(x_1\mp ix_2)^k v^{\nu}.$
\end{proof}

\begin{rem}\label{rspinor}
In (\ref{spinor}) above, we realize the space $\bS=\bS^{\pm}_{2n}$ inside the Clifford algebra $\bC_{2n}$ as 
$$
\bS=\La^{s}(w_1,\ldots,w_n)I\text{\ \ \ with\ $s=\pm$.}
$$
It is not difficult to find generators of 1-dimensional pieces $\bS^{\nu}$ of $\bS$. Indeed, 
we have that
$$\bS^{\pm}=\La^{\pm}(w_1,\ldots,w_{n-1})I^{\pm}$$
where, for $s=+,$ we put $I^+=I$ and $I^-=w_nI$ and, for $s=-,$ obviously $I^+=w_nI$ and $I^-=I.$   
Hence, by induction on $j$, we deduce easily that, for $s,t\in\{\pm\}$ and $\nu=(\pod\nu,s,t)\in\cS_j,$ we have that 
$$\bS^{\nu}=\La^{t}(w_1,\ldots,w_{n-j})I^{\nu}$$
where we put $I^{(\pod\nu,+,+)}=I^{(\pod\nu,+)},$ $I^{(\pod\nu,+,-)}=w_{n-j+1}I^{(\pod\nu,+)}$, 
$I^{(\pod\nu,-,+)}=w_{n-j+1}I^{(\pod\nu,-)}$ and $I^{(\pod\nu,-,-)}=I^{(\pod\nu,-)}.$  
In particular, we have that $\bS^{\nu}\simeq\bS^{t}_{2(n-j)}$. 
Finally, for each $\nu\in\cS^m=\cS_{n-1},$ the 1-dimensional piece $\bS^{\nu}$ is generated by the element
\begin{equation}\label{generators}
v^{\nu}
=\left\{
\begin{array}{ll}
I^{\nu},&\ \ \ \nu=(\pod\nu,+);\medskip\\{}
w_1I^{\nu},&\ \ \ \nu=(\pod\nu,-).
\end{array}
\right. 
\end{equation}
It is easy to see that\smallskip\\
for $\bS=\bS^+_4,$ we have that $v^+=I$ and $v^-=w_1w_2I$;\smallskip\\
for $\bS=\bS^-_4,$ we have that $v^+=w_2I$ and $v^-=w_1I$;\smallskip\\
for $\bS=\bS^+_6,$ $v^{++}=I$, $v^{+-}=w_1w_2I$, $v^{-+}=w_2w_3I$, $v^{--}=w_1w_3I$;\smallskip\\
for $\bS=\bS^-_6,$ $v^{++}=w_3I$, $v^{+-}=w_1w_2w_3I$, $v^{-+}=w_2I$ and $v^{--}=w_1I$.
\end{rem}

In fact, we have constructed a~complete orthogonal Appell system for the complex Hilbert space
$L^2(\bB_m,\bS)\cap \Ker \pa$ of $L^2$-integrable monogenic functions $g:\bB_m\to\bS$.
Indeed, using Theorem \ref{GT_spinor}, we easily obtain the following result.

\begin{thm}\label{tTaylor_spinor} Let $m\geq 3$ and let $\bS$ be a~basic spinor representation for $Spin(m)$. 

\begin{itemize}
\item[(a)] Then an orthogonal basis of the space $L^2(\bB_m,\bS)\cap \Ker \pa$ is formed by the polynomials
$f^{\nu}_{k,\mu}$ for $k\in\bN_0$, $\mu\in J^m_k$ and $\nu\in\cS^m$.
Here the basis elements $f^{\nu}_{k,\mu}$ are defined in Theorem \ref{GT_spinor}.

\item[(b)]
Each function $g\in L^2(\bB_m,\bS)\cap \Ker \pa$ has a~unique orthogonal series expansion
\begin{equation}\label{taylor_spinor}
g = \sum_{k=0}^{\infty}\sum_{\nu\in\cS^m} \sum_{\mu\in J^m_k} \mathbf{t}^{\nu}_{k,\mu}(g)\;f^{\nu}_{k,\mu}
\end{equation}
for some complex coefficients $\mathbf{t}^{\nu}_{k,\mu}(g)$. 

In addition, let $g=\sum_{\nu\in\cS^m} g^{\nu} v^{\nu}$ for some complex functions $g^{\nu}$ on $\bB_m$.
Then, for $\mu=(k_{m-1},\ldots,k_2)\in J^m_k$ and $\nu=(\pod\nu,\pm)\in\cS^m$, we have that
\begin{equation*}
\mathbf{t}^{\nu}_{k,\mu}(g)
=\frac{1}{k!}\;\pa_{\pm}^{k_2}\;\pa^{k_3-k_2}_{x_3}\cdots\pa^{k-k_{m-1}}_{x_m}  g^{\nu}(x)|_{x=0}
\end{equation*}
with $\pa_{\pm}=(1/2)(\pa_{x_1}\pm i\pa_{x_2})$.
\end{itemize}
\end{thm}

For a~function $g\in L^2(\bB_m,\bS)\cap \Ker \pa$, we call the orthogonal series expansion \eqref{taylor_spinor} its generalized Taylor series.

\begin{rem}
Of course, there is a~close connection between the generalized Taylor series expansions from Theorem \ref{tTaylor_CA} and Theorem \ref{tTaylor_spinor}.   
Indeed, we can always realize the spinor space $\bS$ inside the Clifford algebra $\bC_m$
and then, for each $g\in L^2(\bB_m,\bS)\cap \Ker \pa$, we have that
$$\mathbf{t}_{k,\mu}(g)=\sum_{\nu\in\cS^m} \mathbf{t}^{\nu}_{k,\mu}(g)\; v^{\nu}.$$ 
\end{rem}

\subsection*{Acknowledgments}

I am grateful to V. Sou\v cek for useful conversations.
The financial support from the grant GA 201/08/0397 is gratefully acknowledged.
This work is also a part of the research plan MSM 0021620839, which is financed by the Ministry of Education of the Czech Republic.


\end{document}